\documentclass[12pt]{article}

\usepackage{amsmath,amsthm,xcolor}
\usepackage[T1]{fontenc}
\usepackage{lmodern}
\usepackage{enumitem}
\usepackage[letterpaper,left=2.1cm,right=2.1cm,top=2.5cm,bottom=2.5cm,footskip=1.2cm]{geometry}
\usepackage{hyperref}
\usepackage[bf,small]{titlesec}
\numberwithin{equation}{section}
\theoremstyle{plain}
\newtheorem{theorem}{Theorem}[section]
\newtheorem{lemma}[theorem]{Lemma}
\newtheorem{corollary}[theorem]{Corollary}

\newtheorem{claim}[theorem]{Claim}

\newtheorem{conjecture}[theorem]{Conjecture}

\newtheorem{observation}[theorem]{Observation}

\newtheorem{proposition}[theorem]{Proposition}
\newtheorem{case}{Case}
\newtheorem{subcase}{Subcase}[case]

\theoremstyle{definition}
\newtheorem{definition}[theorem]{Definition}

\title{Connectivity keeping paths in digraphs}

\author{Hojin Chu\thanks{Korea Institute for Advanced Study, Seoul 02455, Republic of Korea. Email: {\tt hojinchu@kias.re.kr}},
\and Boram Park\thanks{Department of Mathematics Education, Seoul National University, Seoul 08826, Republic of Korea. Email: {\tt borampark@snu.ac.kr} },
\and Homoon Ryu\thanks{Department of Mathematics Education, Seoul National University, Seoul 08826, Republic of Korea. Email: {\tt ryuhomoon@ajou.ac.kr}}
}

\date{}

\begin{document}

\maketitle

\begin{abstract}
Mader conjectured that every $k$-strong digraph $D$ with minimum semidegree $\delta^0(D)\ge 2k+m-1$ contains a dipath $P$ of order $m$ such that $D-V(P)$ remains $k$-strong.
For $k=1$, he obtained the weaker bound $\delta^0(D)\ge 2m$.
We confirm the conjecture for $k=1$ by showing that the sharp bound $\delta^0(D)\ge m+1$ suffices.
As a consequence, we show that for every integer $m\ge2$, every strongly connected digraph $D$ with $\delta^0(D)\ge\max\{2,m-1\}$ contains a dipath $P$ of order $m$ such that $D-A(P)$ is strongly connected.
\end{abstract}

\section{Introduction}\label{sec:intro}

Chartrand, Kaugars, and Lick~\cite{CKL72} proved that every $k$-connected graph $G$ with $\delta(G)\ge\lfloor 3k/2\rfloor$ has a vertex $v$ such that $G-v$ remains $k$-connected.
Mader~\cite{Mader10} extended this result by proving that every $k$-connected graph $G$ with $\delta(G)\ge\lfloor 3k/2\rfloor+m-1$ contains a path $P$ of order $m$ such that $G-V(P)$ remains $k$-connected.
Problems of finding a prescribed subgraph whose deletion preserves connectivity are commonly referred to as connectivity keeping subgraph problems; see~\cite{Mader05, TM26} for surveys and further results.

The directed analogue of the result of Chartrand, Kaugars, and Lick is due to Mader~\cite{Mader91}.
He proved that every $k$-strong digraph $D$ with minimum semidegree $\delta^0(D)\ge 2k$ has a vertex $x$ such that $D-x$ remains $k$-strong, and that the bound $2k$ is best possible.
Later, Mader~\cite{Mader12} proposed the following conjecture for connectivity keeping paths in digraphs.

\begin{conjecture}[Mader~{\cite[Conjecture~2]{Mader12}}]\label{conj:mader-digraph}
Let $k$ and $m$ be positive integers.
Every $k$-strong digraph $D$ with $\delta^0(D)\ge 2k+m-1$ contains a dipath $P$ of order $m$ such that $D-V(P)$ remains $k$-strong.
\end{conjecture}

Mader~\cite{Mader12} observed that the bound $2k+m-1$ in Conjecture~\ref{conj:mader-digraph} is best possible for all $k$ and $m$.
To the best of our knowledge, the conjecture has only been verified for $m=1$ with arbitrary $k$ and for $(k,m)=(1,2)$ in \cite{Mader91, Mader12}.
In a related direction, Tian, Lai, Xu, and Meng~\cite{TLXM19} showed that, when $k=1$, the same semidegree bound guarantees connectivity keeping oriented stars and double-stars of order $m$.

For $k=1$ and arbitrary $m$, Mader proved the following weaker result for dipaths.

\begin{theorem}[Mader~{\cite[Proposition~1]{Mader12}}]\label{thm:mader}
Let $m$ be a positive integer.
Every strongly connected digraph $D$ with $\delta^0(D)\ge 2m$ contains a dipath $P$ of order $m$ such that $D-V(P)$ is strongly connected.
\end{theorem}
Mader further observed that the first step of his proof, which constructs an initial dipath of order $m+2$, also works under the weaker assumption $\delta^0(D)\ge m+1$. However, he noted that the next step, finding another dipath of order $m$ avoiding the vertices absorbed in the enlargement, is problematic. He also remarked that finding such a dipath ``seems not so easy'' and pointed out that the complement of the particular enlargement arising in his proof may contain no dipath of order $m$~\cite[p.~329]{Mader12}.

Our main result overcomes this obstacle through a refined extremal argument. We reduce the bound in Theorem~\ref{thm:mader} from $2m$ to $m+1$, thereby confirming Conjecture~\ref{conj:mader-digraph} for $k=1$.

\begin{theorem}\label{thm:main}
Let $m$ be a positive integer.
Every strongly connected digraph $D$ with $\delta^0(D)\ge m+1$ contains a dipath $P$ of order $m$ such that $D-V(P)$ is strongly connected.
\end{theorem}

The bound in Theorem~\ref{thm:main} is best possible, as shown by the following construction, which is implicit in Mader~\cite{Mader12}.
For $n\ge 2$, let $X_1,\ldots,X_{2n}$ be pairwise vertex-disjoint complete digraphs such that $|V(X_i)|=1$ for odd $i$ and $|V(X_i)|=m$ for even $i$.
Define $D_{m,n}$ from their disjoint union by adding all arcs from $V(X_i)$ to $V(X_{i+1})$ for every $i$, where the indices are taken modulo $2n$.
Then $D_{m,n}$ is strongly connected.
Moreover, every vertex has indegree and outdegree exactly $m$, so $\delta^0(D_{m,n})=m$.
Let $P$ be a dipath of order $m$ in $D_{m,n}$.
If $P$ contains the vertex of $X_i$ for some odd $i$, then $D_{m,n}-V(P)$ is not strongly connected because $P$ cannot contain all vertices of either $X_{i-1}$ or $X_{i+1}$, while every dipath from a vertex of $X_{i-1}$ to a vertex of $X_{i+1}$ contains the vertex of $X_i$.
If $P$ contains no vertex in any odd-indexed set, then $P$ is contained in one even-indexed set and hence contains all its vertices, again leaving a digraph that is not strongly connected.

The problem of preserving connectivity after deleting the edges of a subgraph has also been studied for undirected graphs.
Hasunuma~\cite{Hasunuma23,Hasunuma25} considered paths, trees, and Hamiltonian cycles in this setting, while matchings have recently been investigated in~\cite{CKP26,Chu26}.
We consider the analogous arc-deletion problem for digraphs and propose the following counterpart of Conjecture~\ref{conj:mader-digraph}.

\begin{conjecture}\label{conj:arc-deletion}
Let $k\ge1$ and $m\ge2$ be integers.
Every $k$-strong digraph $D$ with $\delta^0(D)\ge\max\{k+1,m-1\}$ contains a dipath $P$ of order $m$ such that $D-A(P)$ is $k$-strong.
\end{conjecture}

The case $m=2$ of Conjecture~\ref{conj:arc-deletion} directly follows from a result of Mader~\cite{Mader85}.
The bound $\max\{k+1,m-1\}$ in Conjecture~\ref{conj:arc-deletion} is best possible.
When $2\le m\le k+1$, the complete digraph on $k+1$ vertices is $k$-strong and has minimum semidegree $k$, but deleting any of its arcs results in a digraph that is not $k$-strong.
When $m\ge k+2$, the complete digraph on $m-1$ vertices is $k$-strong and has minimum semidegree $m-2$, but contains no dipath of order $m$.

We next relate the two conjectures.
Conjecture~\ref{conj:mader-digraph}, if true, would imply an arc-deletion result with a weaker semidegree bound.
Since Theorem~\ref{thm:main} establishes Conjecture~\ref{conj:mader-digraph} for $k=1$, we obtain the following result, thereby confirming Conjecture~\ref{conj:arc-deletion} in this case.

\begin{corollary}\label{cor:arc-deletion}
Let $m\ge2$ be an integer.
Every strongly connected digraph $D$ with $\delta^0(D)\ge\max\{2,m-1\}$ contains a dipath $P$ of order $m$ such that $D-A(P)$ is strongly connected.
\end{corollary}

The rest of the paper is organized as follows.
Section~\ref{sec:preliminaries} introduces notation and preliminary tools.
Section~\ref{sec:proof} is devoted to the proof of Theorem~\ref{thm:main}, and Section~\ref{sec:arc-deletion} establishes the arc-deletion results leading to Corollary~\ref{cor:arc-deletion}.

\section{Notation and preliminaries}\label{sec:preliminaries}

We use standard terminology for digraphs and assume throughout that all digraphs are finite, loopless, and have no parallel arcs, while oppositely directed arcs are allowed.
For $v\in V(D)$, let $N_D^+(v)$ and $N_D^-(v)$ denote its out- and in-neighborhoods, and let $d_D^+(v)=|N_D^+(v)|$ and $d_D^-(v)=|N_D^-(v)|$.
We write $\delta^+(D)=\min_{v\in V(D)}d_D^+(v)$, $\delta^-(D)=\min_{v\in V(D)}d_D^-(v)$, and $\delta^0(D)=\min\{\delta^+(D),\delta^-(D)\}$, where the last parameter is the \emph{minimum semidegree} of $D$.
For $X\subseteq V(D)$, let $D[X]$ be the subdigraph induced by $X$, and write $D-X:=D[V(D)\setminus X]$.
For $A\subseteq A(D)$, we write $D-A$ for the subdigraph of $D$ obtained by deleting all arcs in $A$.
Recall that $D$ is \emph{$k$-strong} if $|V(D)|\ge k+1$ and $D-X$ is strongly connected for every $X\subseteq V(D)$ with $|X|<k$.

For a dipath $P=v_1v_2\cdots v_n$ of order $n$, its \emph{length} is $\ell(P)=n-1$, and its \emph{initial} and \emph{terminal} vertices are $v_1$ and $v_n$, respectively.
For $1\le i\le j\le n$, we write $P[v_i,v_j]$ for the subdipath $v_i v_{i+1}\cdots v_j$.
We regard an arc $(u,v)$ as a dipath of order $2$.
If $P_1,\ldots,P_r$ are dipaths such that the terminal vertex of $P_i$ is the initial vertex of $P_{i+1}$ for each $i<r$, and no other vertex belongs to more than one $P_i$, then we write $P_1+\cdots+P_r$ for their concatenation. In particular, $\ell(P_1+\cdots+P_r)=\sum_{i=1}^r\ell(P_i)$.
For disjoint sets $S,T\subseteq V(D)$, an \emph{$(S,T)$-dipath} is a dipath whose initial vertex is its only vertex in $S$ and whose terminal vertex is its only vertex in $T$.

For a strong component $C$ of $D$, we call $C$ a \emph{sink component} if no arc of $D$ leaves $V(C)$, and a \emph{source component} if no arc of $D$ enters $V(C)$. 

The \emph{reverse} $D^{\mathrm{rev}}$ is obtained from $D$ by reversing every arc.
Reversing all arcs in a digraph preserves the minimum semidegree and strong connectivity, maps each dipath to a dipath of the same order, and interchanges sink and source components.

Let $P$ be a dipath of $D$ with initial vertex $x$ and terminal vertex $y$.
For a subdigraph $H$ of $D$, $P$ is said to be \emph{$H$-leaving} if $V(P)\cap V(H)=\{x\}$, and \emph{$H$-entering} if $V(P)\cap V(H)=\{y\}$.
Reversing all arcs in a digraph interchanges $H$-leaving and $H$-entering dipaths.

The following endpoint-counting lemma requires only a minimum outdegree condition and will be used repeatedly in the proof of Theorem~\ref{thm:main}.

\begin{lemma}\label{lem:count}
Let $m$ be a positive integer, let $D$ be a digraph with $\delta^+(D)\ge m+1$, and let $P=v_1v_2\cdots v_m$ be a dipath of order $m$ in $D$.
For a nonempty subset $X$ of $V(D)\setminus V(P)$ such that $N_D^+(x)\subseteq X\cup V(P)$ for every $x\in X$, suppose that $R$ is a longest dipath in $D[X]$ with a prescribed initial vertex, and let $r$ be its terminal vertex.
Then \[ 
\ell(R) \ge a,
\] 
where $a:=\min\bigl(\{i\in\{1,\ldots,m\}:v_i\in N_D^+(r)\}\cup\{m+1\}\bigr)$.
\end{lemma}

\begin{proof}
By the maximality of $R$, $r$ has no out-neighbor in $X\setminus V(R)$. 
Therefore $N_D^+(r)\subseteq V(R)\cup V(P)$.
By the definition of $a$, we have $|N_D^+(r)\cap V(P)|\le m-a+1$.
Since $D$ has no loops, $r\notin N_D^+(r)$, and hence $|N_D^+(r)\cap V(R)|\le\ell(R)$.
It follows that
\[
m+1\le\delta^+(D)\le d_D^+(r)\le\ell(R)+(m-a+1),
\]
which gives $\ell(R)\ge a$.
\end{proof}

\section{Proof of Theorem~\ref{thm:main}}\label{sec:proof}

Throughout this section, let $m$ be a positive integer and let $D$ be a strongly connected digraph with $\delta^0(D)\ge m+1$.
The terminal vertex of a longest dipath in $D$ has all its out-neighbors on the dipath, and hence $D$ contains a dipath of order at least $\delta^+(D)+1\ge m+2$.
In particular, $|V(D)|\ge m+2$, and $D-V(P)$ is nonempty for every dipath $P$ of order $m$.

\begin{definition}\label{def:$m$-max-pair}
An \emph{$m$-max pair} of $D$ is a pair $(P,H)$ such that $P$ is a dipath of order $m$, $H$ is a strong component of $D-V(P)$, and $|V(H)|$ is maximum among all such pairs.
\end{definition}

An $m$-max pair exists by the preceding observation.
Moreover, $(P,H)$ is an $m$-max pair of $D$ if and only if the corresponding reversed pair is an $m$-max pair of $D^{\mathrm{rev}}$.

The following lemma gives a piece of what we referred to in Section~1 as the first step observed by Mader \cite{Mader12}. Since Mader did not provide the details, we include a proof here.

\begin{lemma}\label{lem:leaving-H}
Let $D$ be a strongly connected digraph with $\delta^0(D)\ge m+1$, and let $(P,H)$ be an $m$-max pair of $D$ such that $V(H)\ne V(D)\setminus V(P)$.
Then $D$ has an $H$-leaving or an $H$-entering dipath of length at least $m+2$.
\end{lemma}
\begin{proof}
Let $P=v_1v_2\cdots v_m$.
Whenever we apply Lemma~\ref{lem:count}, we denote the resulting longest dipath by $R$, its terminal vertex by $r$, and the index defined in that lemma by $a$, so that $\ell(R)\ge a$.

\setcounter{case}{0}
\begin{case}
$H$ is not both a sink component and a source component of $D-V(P)$.
\end{case}
By symmetry under arc reversal, we may assume that $H$ is not a sink component of $D-V(P)$.
Then there exists a sink component $C'\ne H$ of $D-V(P)$ that is reachable from $H$.
Let $Q_0$ be a shortest dipath in $D-V(P)$ from $V(H)$ to $V(C')$, and let $c$ be its terminal vertex.
Then $Q_0$ is a $(V(H),V(C'))$-dipath and $\ell(Q_0)\ge 1$.
Since no arc of $D-V(P)$ leaves $V(C')$, we may apply Lemma~\ref{lem:count} with $X=V(C')$ and prescribed initial vertex $c$.
Thus the concatenation $Q_0+R$ is an $H$-leaving dipath.
If $a=m+1$, then $\ell(Q_0+R)=\ell(Q_0)+\ell(R)\ge 1+(m+1)=m+2$, so $Q_0+R$ is a desired dipath.
If $a\le m$, then $v_a\in N^+_D(r)$, so $Q_0+R+(r,v_a)+P[v_a,v_m]$ is an $H$-leaving dipath of length at least $1+a+1+(m-a)=m+2$.

\begin{case}
 $H$ is both a sink component and a source component of $D-V(P)$.
\end{case}
Let $C=D-(V(H)\cup V(P))$.
Since $V(H)\ne V(D)\setminus V(P)$, $V(C)\ne\emptyset$.
No arc of $D$ joins $V(H)$ and $V(C)$ in either direction.
Consequently, $N_D^+(x)\subseteq V(C)\cup V(P)$ for every $x\in V(C)$.
Since $D$ is strongly connected and $H$ is a sink component of $D-V(P)$, there exists an arc $(h^*,v_p)$ of $D$ for some $h^*\in V(H)$ and $p\in\{1,\ldots,m\}$.

\begin{subcase}
 $(v_m,h)\in A(D)$ for some $h\in V(H)$. 
\end{subcase}
We apply Lemma~\ref{lem:count} with $X=V(C)$ and an arbitrary prescribed initial vertex.
If $a\le m$, then $v_a\in N_D^+(r)$, so $L:=R+(r,v_a)+P[v_a,v_m]+(v_m,h)$ is an $H$-entering dipath with $\ell(L)=\ell(R)+1+(m-a)+1\ge m+2$.
Suppose $a=m+1$.
Then $\ell(R)\ge m+1$, so the initial segment $P^*$ of $R$ of order $m$ is well-defined and satisfies $V(P^*)\subseteq V(C)$.
Note that the arcs $(h^*,v_p)$ and $(v_m,h)$, together with the strong connectivity of $H$, imply that $D[V(H)\cup V(P[v_p,v_m])]$ is strongly connected.
Since $V(H)\cup V(P[v_p,v_m])$ is disjoint from $V(P^*)$, $D-V(P^*)$ has a strong component $W$ such that $W$ contains $V(H)\cup V(P[v_p,v_m])$.
This contradicts the choice of the $m$-max pair $(P,H)$ since $|V(H)\cup V(P[v_p,v_m])|=|V(H)|+(m-p+1)>|V(H)|$.

\begin{subcase}
$N_D^+(v_m)\cap V(H)=\emptyset$.
\end{subcase}
Then $N_D^+(v_m)\subseteq V(P)\cup V(C)$.
Since $|N_D^+(v_m)\cap V(P)|\le m-1<m+1\le d_D^+(v_m)$, there is an arc $(v_m,r_0)$ of $D$ for some $r_0\in V(C)$.
We apply Lemma~\ref{lem:count} with $X=V(C)$ and prescribed initial vertex $r_0$.
The sets $\{h^*\}$, $V(P)$, and $V(R)$ are pairwise disjoint, and hence $L:=(h^*,v_p)+P[v_p,v_m]+(v_m,r_0)+R$ is an $H$-leaving dipath of length $(m-p)+2+\ell(R)$.
If $a\ge p$, then $\ell(L)\ge(m-p)+2+a\ge m+2$.
Suppose that $a<p$.
Then $v_a\in N_D^+(r)$ and $V(P[v_a,v_{p-1}])$ is disjoint from $V(P[v_p,v_m])\cup V(R)\cup\{h^*\}$.
Consequently, $L':=L+(r,v_a)+P[v_a,v_{p-1}]$ is an $H$-leaving dipath, and $\ell(L')=(m-p)+2+\ell(R)+1+(p-1-a)\ge m+2$.
\end{proof}

The preceding lemma proves slightly more than is needed, since an $H$-leaving or an $H$-entering dipath of length $m+1$ suffices in the proof of Theorem~\ref{thm:main}. We are now ready to prove our main theorem.

\begin{proof}[Proof of Theorem~\ref{thm:main}]
Let $D$ be a strongly connected digraph with $\delta^0(D)\ge m+1$ and let $(P,H)$ be an $m$-max pair of $D$.
If $V(H)=V(D)\setminus V(P)$, then $D-V(P)$ is strongly connected, so we are done.
Suppose to the contrary that $V(H)\ne V(D)\setminus V(P)$.
By Lemma~\ref{lem:leaving-H}, $D$ contains an $H$-leaving or an $H$-entering dipath of length at least $m+2$.
Since the conclusion of the theorem is invariant under reversing all arcs, we may assume that $D$ contains an $H$-leaving dipath of length at least $m+2$.

Among all pairs $(Q,Z)$ such that $Q$ is an $H$-leaving dipath of order $m+2$ and $Z$ is a $(V(Q)\setminus V(H),V(H))$-dipath, we choose one for which $\ell(Z)$ is maximum, and subject to this, for which the initial vertex of $Z$ has the largest possible index on $Q$.
Indeed, since $D$ is strongly connected, there exists a $(V(Q)\setminus V(H),V(H))$-dipath.
Write $Q=q_0q_1\cdots q_{m+1}$, so that $V(Q)\cap V(H)=\{q_0\}$.
Let $q_s$ be the initial vertex of $Z$, and let $z^*\in V(H)$ be its terminal vertex.
By the definition of a $(V(Q)\setminus V(H),V(H))$-dipath,
\begin{equation}\label{eq:Z-meets}
V(Z)\cap V(Q)\subseteq\{q_s,q_0\}
\qquad\text{and}\qquad
V(Z)\cap V(H)=\{z^*\}.
\end{equation}
Let $W=V(H)\cup V(Z)\cup V(Q)$.
Since $q_1\notin V(H)$, let $R=q_1r_2\cdots r_t$ be a longest dipath with initial vertex $q_1$ in $D-(W\setminus\{q_1\})$.
The case $s=1$, in which $q_1\in V(Z)$, is allowed.
By the choice of $R$,
\begin{equation}\label{eq:R-meets}
V(R)\cap W=\{q_1\}.
\end{equation}

\begin{claim}\label{clm:containment}
$N_D^+(r_t)\subseteq V(R)\cup\{q_1,\ldots,q_{m+1}\}$.
\end{claim}

\begin{proof}
Suppose that $y\in N_D^+(r_t)$ for some $y\in V(H)\cup(V(Z)\setminus\{q_s\})$.
Let
\[
B=
\begin{cases}
V(H)\cup V(R),&\text{if }y\in V(H),\\
V(H)\cup V(R)\cup V(Z[y,z^*]),&\text{if }y\in V(Z)\setminus\{q_s\}.
\end{cases}
\]
Since $H$ is strongly connected, $D[B]$ is strongly connected.
By~\eqref{eq:R-meets}, we have $V(R)\cap V(H)=\emptyset$, and hence $|B|\ge|V(H)|+t>|V(H)|$.
Moreover, $B\cap V(Q)\subseteq\{q_0,q_1\}$ by~\eqref{eq:Z-meets} and~\eqref{eq:R-meets}.
Thus $q_2q_3\cdots q_{m+1}$ is a dipath of order $m$ whose vertex set is disjoint from $B$.
It follows that $D-\{q_2,q_3,\ldots,q_{m+1}\}$ has a strong component of order greater than $|V(H)|$, which contradicts the maximality of $|V(H)|$.
Therefore $N_D^+(r_t)\cap\bigl(V(H)\cup(V(Z)\setminus\{q_s\})\bigr)=\emptyset$.
By the maximality of $R$, $r_t$ has no out-neighbor outside $V(R)$ in $D-(W\setminus\{q_1\})$.
Hence, by~\eqref{eq:Z-meets}, we have $N_D^+(r_t)\subseteq V(R)\cup\{q_1,\ldots,q_{m+1}\}$.
\end{proof}

Let $W_0:=V(H)\cup V(Z)\cup V(Q[q_0,q_s])$ be a subset of $W$.
Since $q_1\in W_0\setminus V(H)$, by the orientations of $Q$ and $Z$, the following is immediate.
\begin{observation}\label{obv:W}
The digraph $D[W_0]$ is strongly connected and $|W_0|>|V(H)|$.
\end{observation}
We now use the strong subdigraph $D[W_0]$ of $D$ to bound $t$ and $|N_D^+(r_t)\cap V(Q)|$.

\begin{claim}\label{clm:outneighbors}
The following statements hold.
\begin{enumerate}[label=\textup{(\roman*)},leftmargin=*,itemsep=1pt]
\item $2\le t\le m$.
\item $\bigl|N_D^+(r_t)\cap V(Q)\bigr|\ge m+3-t$.
\end{enumerate}
\end{claim}

\begin{proof}
If $t=1$, then $r_t=q_1$ and $V(R)=\{q_1\}$, so Claim~\ref{clm:containment} implies $d_D^+(r_t)\le m$, which contradicts $\delta^+(D)\ge m+1$.
Thus $t\ge 2$.
To show $t \le m$, suppose to the contrary that $t\ge m+1$.
Then $R':=r_2r_3\cdots r_{m+1}$ is a dipath of order $m$ whose vertex set is disjoint from $W_0$ by~\eqref{eq:R-meets}.
By Observation~\ref{obv:W}, the digraph $D-V(R')$ has a strong component of order greater than $|V(H)|$, which contradicts the choice of the $m$-max pair $(P,H)$.
Therefore (i) holds.

By Claim~\ref{clm:containment}, every out-neighbor of $r_t$ outside $V(Q)$ belongs to $\{r_2,\ldots,r_{t-1}\}$.
Since $d_D^+(r_t)\ge m+1$, it follows that $\bigl|N_D^+(r_t)\cap V(Q)\bigr|\ge d_D^+(r_t)-(t-2)\ge m+3-t$.
\end{proof}

By Claim~\ref{clm:outneighbors}(i), we have $m\ge2$.

\smallskip
\begin{claim}\label{clm:rotation}
$(r_t,q_j)\notin A(D)$ for every $j$ satisfying either
\[
2\le j\le\min\{s,t\} \qquad \text{or} \qquad s+1\le j\le t+1.
\]
\end{claim}

\begin{proof}
Suppose that $(r_t,q_j)\in A(D)$ for some $j$ with $2\le j\le t+1$. Set $b:=j+m-t$.
Since $t\le m$ and $j\le t+1$, $j\le b\le m+1$.
Since $t\ge2$, the dipath $R[r_2,r_t]$ is well-defined and its vertices are disjoint from $V(Q)$ by~\eqref{eq:R-meets}.
Let $R_0=R[r_2,r_t]+(r_t,q_j)+Q[q_j,q_b]$.
Then $|V(R_0)|=(t-1)+(b-j+1)=m$ and $V(R_0)$ is disjoint from $V(H)$.

We first consider the case $s+1\le j\le t+1$.
All vertices $q_j,\ldots,q_b$ have indices greater than $s$.
It follows from~\eqref{eq:Z-meets} and~\eqref{eq:R-meets} that $V(R_0)\cap W_0=\emptyset$.
By Observation~\ref{obv:W}, the digraph $D-V(R_0)$ has a strong component of order greater than $|V(H)|$, which contradicts the maximality of $|V(H)|$.

Now we consider the case $2\le j\le\min\{s,t\}$.
Then $b\le m$ and $Q':=Q[q_0,q_1]+(q_1,r_2)+R_0$ is an $H$-leaving dipath of order $m+2$ satisfying
\begin{equation}\label{eq:Qprime-in-Q}
V(Q')\cap V(Q)=\{q_0,q_1\}\cup\{q_j,\ldots,q_b\}.
\end{equation}
Since $s\ge j\ge2$, we have $q_1\ne q_s$, and hence $V(R)\cap V(Z)=\emptyset$ by~\eqref{eq:Z-meets} and~\eqref{eq:R-meets}.

Suppose $b<s$.
Let $Z'=Q[q_b,q_s]+Z$.
The vertices $q_{b+1},\ldots,q_s$ are disjoint from $V(Q')$ by~\eqref{eq:Qprime-in-Q}, while $V(Z)\cap V(Q')\subseteq\{q_0\}\subseteq V(H)$ by~\eqref{eq:Z-meets}.
Since $V(R)\cap V(Z)=\emptyset$, $Z'$ is a $(V(Q')\setminus V(H),V(H))$-dipath.
However, $\ell(Z')=(s-b)+\ell(Z)>\ell(Z)$, which contradicts the maximality of $\ell(Z)$.

Suppose $b\ge s$.
Since $j\le s\le b$, the vertex $q_s$ lies on $Q'$.
Moreover,~\eqref{eq:Z-meets},~\eqref{eq:Qprime-in-Q}, and $V(R)\cap V(Z)=\emptyset$ imply that $V(Z)\cap(V(Q')\setminus V(H))=\{q_s\}$ and $V(Z)\cap V(H)=\{z^*\}$.
Thus $(Q',Z)$ belongs to the collection of pairs over which $(Q,Z)$ was chosen.
Relabel $Q'=q_0'q_1'\cdots q_{m+1}'$, and let $s'$ be the index such that $q_{s'}'=q_s$.
Since $j\le t$, $s'=(t+1)+(s-j)>s$.
This contradicts the secondary maximality condition in the choice of $(Q,Z)$.
\end{proof}

By Claim~\ref{clm:outneighbors}(i), we have $2\le t\le m$.
Every $j\in\{2,\ldots,t\}$ satisfies either $j\le\min\{s,t\}$ or $s+1\le j\le t+1$.
Thus Claim~\ref{clm:rotation} implies that none of the $t-1$ vertices $q_2,\ldots,q_t$ is an out-neighbor of $r_t$.
Then, by Claim~\ref{clm:containment}, $N_D^+(r_t)\cap V(Q)\subseteq \{q_1, q_{t+1}, \ldots, q_{m+1}\}$ and so $|N_D^+(r_t)\cap V(Q)|\le 1+(m+1-t)=m+2-t$, which contradicts Claim~\ref{clm:outneighbors}(ii).
Hence, we have shown that $V(H)=V(D)\setminus V(P)$ and $D-V(P)$ is strongly connected.
\end{proof}

\section{From vertex-deletion to arc-deletion}\label{sec:arc-deletion}

The following proposition makes precise the relationship between Conjectures~\ref{conj:mader-digraph} and~\ref{conj:arc-deletion}.
It is folklore that for a $k$-strong subdigraph $H$ of a digraph $D$, if every vertex in $V(D)\setminus V(H)$ has at least $k$ out-neighbors and at least $k$ in-neighbors in $V(H)$, then $D$ is $k$-strong.

\begin{proposition}\label{prop:arc-deletion}
Let $k\ge1$ and $m\ge3$ be integers. 
Suppose that Conjecture~\ref{conj:mader-digraph} holds for $k$ and $m-2$.
Then every $k$-strong digraph $D$ with $\delta^0(D)\ge 2k+m-3$ contains a dipath $P$ of order $m$ such that $D-A(P)$ is $k$-strong.
\end{proposition}

\begin{proof}
By Conjecture~\ref{conj:mader-digraph} for $k$ and $m-2$, there is a dipath $Q=q_1q_2\cdots q_{m-2}$ of order $m-2$ such that $H:=D-V(Q)$ is $k$-strong.
For every $q\in V(Q)$, we have
\[
\begin{aligned}
|N_D^+(q)\cap V(H)|
 &\ge d_D^+(q)-(|V(Q)|-1)\ge (2k+m-3)-(m-3)=2k,\\
|N_D^-(q)\cap V(H)|
 &\ge d_D^-(q)-(|V(Q)|-1)\ge (2k+m-3)-(m-3)=2k.
\end{aligned}
\]
Thus we may choose distinct vertices
\[
x\in N_D^-(q_1)\cap V(H)
\qquad\text{and}\qquad
y\in N_D^+(q_{m-2})\cap V(H).
\]
Then $P:=xq_1q_2\cdots q_{m-2}y$ is a dipath of order $m$.

Note that no arc of $P$ has both ends in $V(H)$. 
Thus $H$ is a $k$-strong subdigraph of $D-A(P)$.
For each vertex $q \in V(Q)$, since at most one arc of $P$ leaves and at most one arc of $P$ enters $q$, 
\[
\begin{aligned}
|N_{D-A(P)}^+(q)\cap V(H)|
&\ge |N_D^+(q)\cap V(H)|-1\ge 2k-1\ge k,\\
|N_{D-A(P)}^-(q)\cap V(H)|
 &\ge |N_D^-(q)\cap V(H)|-1\ge 2k-1\ge k.
\end{aligned}
\]
Therefore $D-A(P)$ is $k$-strong.
\end{proof}

\begin{proof}[Proof of Corollary~\ref{cor:arc-deletion}]
If $m=2$, the result follows from Mader's result~\cite{Mader85}.
If $m\ge3$, the result follows from Theorem~\ref{thm:main} and Proposition~\ref{prop:arc-deletion} with $k=1$.
\end{proof}

When $m=3$, the hypothesis of Proposition~\ref{prop:arc-deletion} is precisely the case $m=1$ of Conjecture~\ref{conj:mader-digraph}, which was proved by Mader~\cite{Mader91}.
Thus Proposition~\ref{prop:arc-deletion} with $m=3$ yields the following corollary.

\begin{corollary}\label{cor:arc-deletion-order-three}
Let $k$ be a positive integer.
Every $k$-strong digraph $D$ with $\delta^0(D)\ge2k$ contains a dipath $P$ of order three such that $D-A(P)$ is $k$-strong.
\end{corollary}

\section*{Acknowledgements}

Hojin Chu was supported by a KIAS Individual Grant (CG101801) at Korea Institute for Advanced Study.
Boram Park was supported by the National Research Foundation of Korea (NRF) grant funded by the Korea government (MSIT) (No.~RS-2025-00523206) and by the New Faculty Startup Fund from Seoul National University.

\bigskip
\noindent\textbf{Declaration of generative AI use.}
\par\smallskip
\noindent 
During the preparation of this work, the authors used OpenAI's GPT-5.6 Sol to explore possible refinements to the proof, check proof details, and improve the exposition.
The problem formulation, main structural ideas, and proof framework were developed by the authors.
All AI-assisted suggestions were independently verified and revised by the authors, who take full responsibility for the manuscript.

\end{document}